\documentclass[12pt,a4paper,twoside]{amsart}%final
\usepackage{version}
\usepackage{amssymb}
\usepackage{graphicx}

\input epsf

\newtheorem{theorem}{Theorem}[section]

\newtheorem{corollary}[theorem]{Corollary}

\theoremstyle{definition}
\newtheorem{definition}[theorem]{Definition}

\theoremstyle{remark}

\numberwithin{equation}{section}

\DeclareMathOperator{\dimH}{dim_H}
\DeclareMathOperator{\dimE}{dim_E}
\DeclareMathOperator{\ddh}{d_H}
\DeclareMathOperator{\dimS}{dim_S}

\DeclareMathOperator{\Graf}{\mathcal G}

\DeclareMathOperator{\hm}{\mathcal H}

\newcommand{\pv}{ : }
\newcommand{\pr}{\mathbb{R}}
\newcommand{\prn}{\mathbb{R}^{2n}}

\newcommand{\Gh}{G_h(n,m)}
\newcommand{\Lagr}{G_h(n,n)}
\newcommand{\Mh}{\mathcal{M}_h(n,m)}
\newcommand{\MhB}{\mathcal{M}_h^1(n,m)}
\newcommand{\He}{\mathbb{H}}
\newcommand{\SG}{\mathbb{G}}
\newcommand{\Hor}{\mathbb{V}}
\newcommand{\Ver}{\mathbb{V}^\perp}

\newcommand{\om}{\omega}

\renewcommand{\phi}{\varphi}
\renewcommand{\tilde}{\widetilde}

\begin{document}

\title[Transversality of isotropic projections]
  {Transversality of isotropic projections, unrectifiability and Heisenberg groups}

\author[Risto Hovila]{Risto Hovila$^1$}
\address{Department of Mathematics and Statistics, P.O.Box 68,
         00014 University of Helsinki, Finland$^1$}
\email{risto.hovila@helsinki.fi}
\thanks{The research was supported by the Finnish Centre of Excellence in Analysis and Dynamics Research and by the Jenny and Antti Wihuri Foundation.}

\begin{abstract}
 We show that the family of $m$-dimensional isotropic projections in $\prn$
is transversal. As an application we show that the
Besicovitch-Federer projection theorem holds for isotropic projections. We also use
transversality to obtain almost sure estimates on the Hausdorff dimension of
isotropic projections of subsets $E \subset \prn$. These results may also be
applied to gain information on the horizontal projections of the Heisenberg
group $\He^n$.
\end{abstract}

\maketitle

\section{Introduction}\label{intro}

Let $\prn$ be equipped with the standard symplectic form $\om:\prn\times\prn\to\pr,
\om(x,y) = \sum_{i=1}^n x_{i+n}y_i - x_iy_{i+n}$. In this note we examine
properties of $m$-dimensional isotropic projections in $\prn$, that is, orthogonal
projections onto $m$-dimensional isotropic subspaces of $\prn$. A linear subspace
$V \subset \prn$ is isotropic if $\om(v,w)=0$ for all $v,w \in V$.
Isotropic subspaces are closely related to horizontal subgroups of the
Heisenberg group $\He^n$. Indeed, $\Hor \subset \He^n$ is a horizontal subgroup,
if and only if $\Hor = V \times \{0\}$ for some isotropic subspace $V \subset \He^n$.
Thus the study of isotropic projections
also yields results on the horizontal projections of $\He^n$.

Our main result is the following theorem.

\begin{theorem}\label{main}
Let $n,m$ be integers such that $0<m\leq n$, let $\Gh$ be the submanifold of the
Grassmannian $G(2n,m)$ consisting of all isotropic
subspaces of $\pr^{2n}$ and denote by $P_V:\pr^{2n} \to V$ the orthogonal projection
onto the $m$-plane $V\in\Gh$. Then the projection family
$\{P_V:\pr^{2n}\to V\}_{V\in \Gh}$ is transversal.
\end{theorem}

This paper is organized as follows:
In Section \ref{preli} we introduce the basic definitions and notation.
In Section \ref{trans_sect} we will prove Theorem \ref{main}.
As an application of the main result we give necessary and sufficient conditions
under which a subset
$E \subset \pr^{2n}$ projects onto a set of measure zero under almost all
$m$-dimensional isotropic projections (Theorem \ref{besfed_isotropic}).
We also give almost sure dimension estimates on the Hausdorff dimension of isotropic
projections of subsets of $\prn$ (Theorem \ref{dim_of_excep}). These estimates
were already proven in \cite[Theorem 1.2]{BFMT} using different methods. We improve
the result by providing estimates on the dimension of exceptional parametres.
The results mentioned above
also yield corollaries concerning the dimension of horizontal projections of subsets
of the Heisenberg group. The applications will be discussed in Section
\ref{appl_sect}.

\section{Preliminaries}\label{preli}

\subsection{Symplectic geometry}

Let $M$ be a manifold of dimension $2n$. A \textit{symplectic form} on $M$ is a
closed non-degenerate $2$-form on $M$. The \textit{standard form} $\om$ on $\prn$ is
defined by
\[
 \om(x,y) = \sum_{i=1}^n x_{i+n}y_i - x_iy_{i+n} = (Jx \, | \, y),
\]
where $(\cdot \, | \, \cdot)$ is the Euclidean inner product on $\prn$ and $J$ is the
$2n \times 2n$-matrix
\[
 J = \left(
\begin{matrix}
  0             & I_{n\times n} \\
 -I_{n\times n} & 0
\end{matrix}
\right).
\]
By a well-known theorem of Darboux, every symplectic form on $M$ is locally
diffeomorphic to the standard form $\om$ on $\prn$. Furthermore, every symplectic
vector space is isomorphic to $(\prn, \om)$. Below we work only on $\prn$ equipped
with the standard form $\om$. For more information on symplectic geometry, see
\cite{ET}.

For a linear subspace $V \subset \prn$, we define its \textit{symplectic orthogonal}
$V^\om$ by
\[
 V^\om = \{ w \pv \om(w,v) = 0 \text{ for all } v \in V\}.
\]
A linear subspace $V$ is said to be \textit{isotropic}, if $V \subset V^\om$, and
\textit{Lagrangian}, if $V = V^\om$. A subspace $V$ can be Lagrangian only when
$\dim V = n$. For integers $0 < m \leq n$, we denote by $G(2n,m)$ the space of all
$m$-dimensional linear subspaces of $\prn$. It is a compact manifold of dimension
$m(2n-m)$. Furthermore, we define the {\it isotropic Grassmannian}
$\Gh$ by
\[
 \Gh = \{V \in G(2n,m) \pv V \text{ is an isotropic subspace of } \prn\}.
\]
In the case $m=n$, $\Lagr$ is called the \textit{Lagrangian Grassmannian}. $\Gh$
is a smooth manifold of dimension $2nm-\frac{m(3m-1)}{2}$,
so for $m>1$ the isotropic Grassmannian $\Gh$ is a
submanifold of $G(2n,m)$ with positive codimension. For $m=1$, the manifolds are
the same, $G_h(n,1)=G(2n,1)$. The isotropic Grassmannian can be endowed with
a natural measure $\mu_{n,m}$ in a similar way
as the usual Grassmannian is endowed with the measure $\gamma_{n,m}$, using unitary
instead of orthogonal matrices. See \cite[Section 2]{BFMT} for more details. 

Next we define local coordinates on the isotropic Grassmannian $\Gh$. We begin by
recalling the definition of local coordinates on the Grassmannian
manifold $G(2n,m)$. Fix an $m$-plane $V \in G(2n,m)$ and choose an orthonormal basis
$\{e_1, \ldots, e_{2n}\}$ of $\prn$ such that $V = <e_1, \ldots, e_m>$. Consider
all linear maps $L(V, V^\perp) = \{L:V \to V^\perp \pv L \text{ linear}\}$. The graph
$\Graf(L) = \{(x,Lx) \pv x\in V\}$ of any such map is an $m$-plane whose intersection
with the $(2n-m)$-plane $V^\perp$ is the zero subspace. Conversely, any $m$-plane
with this property is the graph of a unique linear map $L:V\to V^\perp$. Using the
basis $\{e_1, \ldots, e_m\}$ of $V$ and the basis $\{e_{m+1}, \ldots, e_{2n}\}$ for
$V^\perp$, $L(V,V^\perp)$ can be identified with $\mathcal M(2n-m,m)$, the space of
all $(2n-m) \times m$ matrices. The $m$-plane associated to a matrix $A = (a_{ij})
\in \mathcal M(2n-m,m)$ is spanned by the vectors
\begin{equation}\label{span_vect}
 e^A_i = e_i + \sum_{k=1}^{2n-m} a_{ki} \, e_{k+m}, \quad i = 1, \ldots, m.
\end{equation}

Define a subset $\Mh$ of all $(2n-m) \times m$ matrices by setting
\begin{equation}\label{isotropic_mat}
\begin{aligned}
 \Mh =\{(a_{ij}) &\in \mathcal M(2n-m,m) \pv a_{(n-m+i)j}=a_{(n-m+j)i} \\
      &+ \sum_{k=1}^{n-m} (a_{kj}a_{(n+k)i} - a_{(n+k)j}a_{ki})
        \text{ for } j < i \leq m, 1\leq j\leq m\}.
\end{aligned}
\end{equation}

The independent coordinates $a_{ij}$ in a matrix $(a_{ij}) \in \Mh$ are
the ones with $j\in\{1,\ldots,m\}$, $i\in\{1,\ldots,n-m+j,n+1,\ldots,2n-m\}$.

$\Mh$ is an embedded submanifold of $\mathcal M(2n-m,m)$ having dimension
$2nm - \frac{m(3m-1)}{2}$. Let $A \in \mathcal M(2n-m,m)$ and denote the
$m$-plane associated to $A$ by $V_A$. Then $V_A \in \Gh$ if and only if
\[
 \om(e_i + \sum_{k=1}^{2n-m} a_{ki} e_{k+m}, e_j + \sum_{k=1}^{2n-m} a_{kj}
  e_{k+m}) = 0 \text{ for all } i,j = 1, \ldots, m
\]
and this is the case precisely when $A \in \Mh$. Thus we can define
coordinates on the isotropic Grassmannian using the matrices $A \in \Mh$.
It is enough to define the local coordinates only around the $m$-plane $V =
<e_1,\ldots,e_m>$, since the group $U(n)$ of unitary
$2n \times 2n$-matrices acts transitively on $\Gh$ by \cite[Lemma 2.2]{BFMT}
and $\omega(gu,gv)=\omega(u,v)$ for all $u,v \in \prn$, $g \in U(n)$.

\subsection{Heisenberg groups}

For an introduction to Heisenberg groups, see \cite{CDPT}. Below we state the
basic facts needed in this paper.
The Heisenberg group $\He^n$ is the unique simply connected, connected nilpotent
Lie group of step two and dimension $2n+1$ with one dimensional centre. As a manifold
$\He^n$ may be identified with $\pr^{2n+1}$. We denote points $p \in \He^n$ in
coordinates as
\[
 p = (z,t) = (z_1, \ldots, z_{2n}, t) \in \prn \times \pr.
\]
The group operation is given by
\[
 p \ast p' = (z,t) \ast (z',t') = (z+z', t+t'+2\om(z,z')),
\]
where $\om$ is the standard symplectic form on $\pr^{2n}$.

The Heisenberg metric $\ddh$ of $\He^n$ can be defined by
\[
 \ddh(p,p') := \Vert p^{-1} \ast p'\Vert_{\rm H}, \; \text{ where }
  \Vert p\Vert_{\rm H} := (\Vert z\Vert^4 + t^2)^{1/4}.
\]
Here $\Vert\cdot\Vert$ denotes the Euclidean norm on $\pr^{2n}$. This metric is
bi-Lipschitz equivalent to the usual Carnot-Caratheodory metric on $\He^n$.
The metric $\ddh$ induces the Euclidean topology, but the properties of the metric
space $(\He^n,\ddh)$ differ significantly from those of the
underlying Euclidean space. For
example, the Hausdorff dimension of $(\He^n, \ddh)$ is $2n+2$. Thus, when speaking
of the metric properties of $\He^n$, we need to specify which metric we are using.
We will denote the Hausdorff measure and Hausdorff dimension with respect to the
Heisenberg metric by $\hm_{\rm H}^s$ and $\dimH$. The Hausdorff measure and dimension
with respect to the Euclidean metric are denoted by $\hm_{\rm E}^s$ and $\dimE$.

In this paper we consider projections onto homogeneous subgroups of $\He^n$. A {\it
homogeneous subgroup} $\SG$ of $\He^n$ is a subgroup which is closed under the
intrinsic dilatations $\delta_s(z,t) = (sz, s^2t), s > 0$. There are two kinds of
homogeneous subgroups of $\He^n$. The {\it horizontal subgroups} are the ones which
are contained in $\prn \times \{0\}$ and the {\it vertical subgroups} are the
ones which contain the $t$-axis $\{0\} \times \pr$. Horizontal subgroups can be
identified with linear subspaces of $\pr^{2n+1}$ which are contained in $\prn
\times \{0\}$. However not every linear subspace of this form is a horizontal
subgroup, only those corresponding to isotropic subspaces $V$ of $\prn$. The
restriction of the Heisenberg metric to a horizontal subgroup coincides with the
Euclidean metric and therefore it is not necessary to specify the metric used in
computing the Hausdorff measure or Hausdorff dimension of a subset of a horizontal
subgroup. In this case we denote the Hausdorff measure and Hausdorff dimension simply
by $\hm^s$ and $\dim$.

Let $\Hor = V \times \{0\}$ be a horizontal subgroup. Consider $\Ver = V^\perp
\times \pr$, where $V^\perp$ is the orthogonal complement of $V$ in $\prn$. Then
$\Ver$ is a vertical subgroup of $\He^n$ and it will be called the vertical subgroup
associated to $V$.
Each point $p \in \He^n$ can be written uniquely as
\[
 p = P_{\Ver}(p) \ast P_\Hor(p),
\]
with $P_{\Ver}(p) \in \Ver$ and $P_\Hor(p) \in \Hor$. This gives rise to a
well-defined {\it horizontal projection}
\[
 P_\Hor: \He^n \to \Hor, \; (z,t) \mapsto P_\Hor(z,t) = (P_V(z), 0),
\]
and a {\it vertical projection}
\[
 P_{\Ver}: \He^n \to \Ver, \; (z,t) \mapsto P_{\Ver}(z,t) = (P_{V^\perp}(z),
 t - 2\om(P_{V^\perp}(z), P_V(z))).
\]

Since there is a one-to-one correspondence between isotropic subspaces of
$\pr^{2n}$ and horizontal subgroups of $\He^n$, these projections may be
parametrized by the isotropic Grassmannian $\Gh$. For more information on the
projections of the Heisenberg group, see \cite{BCFMT} and \cite{BFMT}.

We denote the family of all horizontal projections onto $m$-dimensional subgroups by
$\mathcal F_h(n,m)$ and the corresponding projections in the Euclidean space
$\prn$ by $\mathcal F_h^{2n}(n,m)$, that is,
\[
  \mathcal F_h(n,m) = \{ P_\Hor: \He^n \to \Hor \pv V \in \Gh \}
\]
and
\[
  \mathcal F_h^{2n}(n,m) = \{ P_V: \prn \to V \pv V \in \Gh \}.
\]

Note that when $m > 1$, the family $\mathcal F_h^{2n}(n,m)$ has dimension
$2nm-\frac{m(3m-1)}{2} < m(2n-m)$, and therefore one cannot apply standard
projection theorems (e.g. Marstrand's projection theorem or Besicovitch-Federer
projection theorem) to obtain dimension results for these projections.

\section{Transversality}\label{trans_sect}

In this section we show that the family $\mathcal F_h^{2n}(n,m)$ of projections is
transversal for every $0 < m \leq n$. We begin with the definition of
transversality.

\begin{definition}\label{transversality} Let $\Lambda\subset\pr^l$ be open.
A family of maps $\{\pi_\lambda:\pr^n \to \pr^m\}_{\lambda \in \Lambda}$
is {\it transversal} if it satisfies the following
conditions for each compact set $K\subset\pr^n$:
\begin{itemize}
\item [(1)] The mapping $\pi:\Lambda\times K\to\pr^m$,
  $(\lambda,x)\mapsto \pi_\lambda(x)$, is continuously differentiable and twice
  differentiable with respect to $\lambda$.
\item [(2)] For $j= 1,2$ there exist constants $C_j$ such that the derivatives
  with respect to $\lambda$ satisfy
  \[
   \Vert D_{\lambda}^j \pi(\lambda, x)\Vert \leq C_j \text{ for all }(\lambda,x)
    \in\Lambda\times K.
  \]
\item [(3)] For all $\lambda\in\Lambda$ and $x$, $y\in K$ with $x\ne y$,
  define
  \[
   \Phi_{x, y}(\lambda) = \frac{\pi_\lambda(x) - \pi_\lambda(y)}{\Vert x-y\Vert }.
  \]
  Then there exists a constant $C_T>0$ such that the property
  \[
   \Vert\Phi_{x, y}(\lambda)\Vert\leq C_T
  \]
  implies that
  \[
   \det \left( D_\lambda \Phi_{x, y}(\lambda) \left(D_\lambda \Phi_{x, y}(\lambda)
    \right)^T \right) \geq C_T^2.
  \]
\item [(4)] There exists a constant $C_L$ such that
  \[
   \Vert D_{\lambda}^2 \Phi_{x, y}(\lambda)\Vert\le C_L
  \]
  for all $\lambda\in\Lambda$ and $x,y\in K$, $x \neq y$.
\end{itemize}
\end{definition}

Since transversality is a local property and
the isotropic Grassmannian $G_h(n,m)$ can be covered by
finitely many coordinate neighbourhoods, we need to show that the projection
family $\mathcal F_h^{2n}(n,m)$ satisfies the above conditions
locally, that is, each $V \in \Gh$ has a
coordinate neighbourhood $U \subset \Gh$ on which 
$\{P_W:\pr^{2n}\to W\}_{W\in U}$ is transversal.

Transversal projection families have many useful properties. For instance, the
Hausdorff dimensions of sets and measures are preserved under almost all projections.
The theory of transversal mappings was extensively studied by Peres and Schlag in
\cite{PS}. A recent result for transversal projection families that we will use in
this paper is the Besicovitch-Federer projection theorem. See \cite{HJJL} for the
proof.

\begin{theorem}\label{besfed}
Let $E\subset\mathbb R^n$ be $\mathcal H^m$-measurable with $\mathcal H^m(E)<\infty$.
Assume that $\Lambda\subset\pr^l$ is open and
$\{\pi_\lambda:\mathbb R^n\to\mathbb R^m\}_{\lambda \in \Lambda}$
is a transversal family of maps. Then $E$ is purely $m$-unrectifiable, if and only if
$\mathcal H^m(\pi_\lambda(E))=0$ for $\mathcal L^l$-almost all $\lambda\in\Lambda$.
\end{theorem}

Fix $0<m\leq n$. To show that 
the transversality condition holds locally,
we fix an $m$-plane $V \in \Gh$, take a coordinate system
around $V$
and show that the transversality conditions hold 
in this coordinate
neighbourhood.

Define a family of projections $\pi: \Mh \times \prn \to \pr^m$ by setting
\[
\begin{aligned}
 \pi(A, x) = \pi_A(x) &= \left( \left( e^A_1 \Big| \; x \right), \ldots,
 \left( e^A_m \Big| \; x \right)\right) \\
 &= \left( x_1 + \sum_{k=1}^{2n-m} a_{k1}x_{k+m}, \ldots, x_m +
 \sum_{k=1}^{2n-m} a_{km}x_{k+m} \right).
\end{aligned}
\]
The projection $\pi_A$ is not quite the same as the orthogonal projection $P_{V_A}$
onto the $m$-plane $V_A$ corresponding to the matrix $A$, since the base $\{e^A_1,
\ldots, e^A_m\}$ of $V_A$ is not orthonormal, but we will see that for $A$
close to $0$ the projections are sufficiently close.

Define a mapping $\Phi: \{(A,x,y)\in\Mh\times\prn\times\prn\pv x\neq y\} \to \pr^m$
by
\[
 \Phi(A, x, y) = \Phi_{x,y}(A) = \frac{\pi_A(x) - \pi_A(y)}{\Vert x-y \Vert}.
\]
Denoting $b = \Vert x-y \Vert^{-1} (x-y) \in S^{2n-1}$ and using the linearity of the
projection $\pi_A$ we see that
\[
 \Phi_{x,y}(A) = \frac{\pi_A(x) - \pi_A(y)}{\Vert x-y \Vert} =
 \frac{\pi_A(x - y)}{\Vert x-y \Vert}
 = \pi_A\left( \frac{x-y}{\Vert x-y \Vert} \right) = \pi_A(b) =: \Phi_b(A).
\]

We will show that for the family defined above it holds that for every $x, y \in
\prn$ with $x \neq y$,
\begin{equation}\label{trans_pi_0}
  \det \left( D_A \Phi_{x, y}(0) \left(D_A \Phi_{x, y}(0)
    \right)^T \right) \geq \frac{1}{2}(1-\Vert\Phi_{x, y}(0)\Vert^2)^m.
\end{equation}

Before we prove this inequality, we show that it
implies that the family
$\mathcal F_h^{2n}(n,m)$ is locally transversal. The family
clearly satisfies the conditions (1), (2) and (4) in Definition
\ref{transversality}, so we need to show that the condition (3) is also valid.
We examine the problem in local coordinates $(U, \phi)$ defined
above. Let $\{v^A_1, \ldots, v^A_m\}$ be the orthonormal basis obtained by applying
the Gram-Schmidt algorithm to the basis $\{e^A_1, \ldots, e^A_m\}$. Then at
$A = 0$, we have for every $b \in S^{2n-1}$ 
\begin{equation}\label{simple_proj}
\left( v^0_i | b \right) = \left( e^0_i | b \right)
\end{equation}
and
\begin{equation}\label{simple_proj_der}
 \partial_{\alpha\beta}\big|_{A=0}  \left( v^A_i | b \right) = \partial_{\alpha\beta}
 \big|_{A=0} \left( e^A_i | b \right)
\end{equation}
for all $i,\beta = 1, \ldots, m,$ and $\alpha \in \{1, \ldots, n-m+\beta, n+1,
\ldots, 2n-m\}$,
where $\partial_{\alpha\beta}$ denotes the partial derivative with respect to
the entry $a_{\alpha\beta}$ in the matrix $A = (a_{\alpha\beta})$.

Defining
$\Phi^{\mathcal F}: \{(A,x,y)\in\Mh\times\prn\times\prn\pv x\neq y\} \to \pr^m$ by
\[
\Phi^{\mathcal F}(A,x,y) = \Phi^{\mathcal F}_{x,y}(V_A)
= \frac{P_{V_A}(x)-P_{V_A}(y)}{\Vert x-y\Vert },
\]
equations \eqref{simple_proj} and \eqref{simple_proj_der} imply that
\[
 \Phi^{\mathcal F}_{x,y}(V_0) = \Phi_{x,y}(0)
\]
and
\[
  \det \left( D_A \Phi_{x, y}(0) \left(D_A \Phi_{x, y}(0) \right)^T \right)
 =\det \left( D_A \Phi^{\mathcal F}_{x, y}(V_0) \left(D_A
  \Phi^{\mathcal F}_{x, y}(V_0) \right)^T \right)
\]
for every $x,y \in \prn$ such that $x \neq y$.
We also use the notation $\tilde\Phi^{\mathcal F}(A,b) =
\tilde\Phi^{\mathcal F}_b(V_A) = P_{V_A}(b)$ for $b \in S^{2n-1}$. Note that
$\tilde\Phi^{\mathcal F}(A,(x-y)/\Vert x-y\Vert ) = \Phi^{\mathcal F} (A,x,y)$. The functions
$\Phi^{\mathcal F}$ and $\tilde\Phi^{\mathcal F}$ are smooth, so defining $\MhB = \Mh
\cap B(0,1)$, we may choose a Lipschitz constant $L_1\geq 1$ for
$\tilde\Phi^{\mathcal F}$ and a Lipschitz constant $L_2\geq 1$ for
$$
(A,b) \mapsto \det \left( D_A \tilde\Phi^{\mathcal F}_{b}(V_A) \left(D_A
  \tilde\Phi^{\mathcal F}_{b}(V_A) \right)^T \right) \text{ on } \MhB \times S^{2n-1}.
$$

Let
\[
0 < C_T \leq 2^{-\frac{m+2}{2}}
\quad \text{ and }
\quad \epsilon = \min \left\{\frac{C_T}{L_1}, \frac{(1-4C_T^2)^m}{4L_2} \right\}.
\]
If $(A,b) \in \left(\Mh \cap B(0,\epsilon)\right) \times S^{2n-1}$ is
such that
  \[
   \Vert\tilde\Phi^{\mathcal F}_b(V_A)\Vert\leq C_T,
  \]
we have by \eqref{simple_proj}
  \[
   \Vert\Phi_b(0)\Vert = \Vert\tilde\Phi^{\mathcal F}_b(V_0)\Vert\leq L_1\epsilon + C_T
   \leq 2C_T
  \]
and by \eqref{trans_pi_0} and \eqref{simple_proj_der}
\[
\begin{aligned}
 \det \left( D_A \tilde\Phi^{\mathcal F}_{b}(V_0) \left(D_A
  \tilde\Phi^{\mathcal F}_{b}(V_0) \right)^T \right) &= \det \left( D_A \Phi_{b}(0)
  \left(D_A \Phi_{b}(0)
    \right)^T \right) \\
  &\geq \frac{1}{2}(1-\Vert\Phi_{b}(0)\Vert^2)^m \\
  &\geq \frac{1}{2}(1-4C_T^2)^m.
\end{aligned}
\]
This implies that
\[
 \det \left( D_A \tilde\Phi^{\mathcal F}_{b}(V_A) \left(D_A
  \tilde\Phi^{\mathcal F}_{b}(V_A) \right)^T \right) \geq \frac{1}{2}(1-4C_T^2)^m
  - L_2\epsilon \geq \frac{1}{4}(1-4C_T^2)^m \geq C_T^2
\]
by the choice of $C_T$ and $\epsilon$.
We have shown that assuming inequality \eqref{trans_pi_0}, every plane $V \in \Gh$
has a coordinate neighbourhood on which the family $\mathcal F_h^{2n}(n,m)$
satisfies the transversality condition.
Next we prove the inequality \eqref{trans_pi_0}.

Recalling \eqref{isotropic_mat}, we see that for $A\in\Mh$ the $j$th component
function of the projection $\pi_A$ has the form
\[
 \begin{aligned}
  \pi_A^j(x) &= x_j + \sum_{k=1}^{2n-m} a_{kj}x_{m+k}\\
  &\begin{aligned}
    = x_j &+ \sum_{k=1}^{n-m+j} a_{kj}x_{m+k} + \sum_{k=n+1}^{2n-m} a_{kj}x_{m+k}\\
    &+ \sum_{k=j+1}^{m} \left(a_{(n-m+j)k} + \sum_{l=1}^{n-m} \left( a_{lj}a_{(l+n)k}
   - a_{lk}a_{(l+n)j} \right) \right) x_{n+k}.
   \end{aligned}
 \end{aligned}
\]
The partial derivatives with respect to the entries in the matrix $A$ are
\[
 \partial_{\alpha\beta}\pi_A^j(x) = \left\{
\begin{tabular}{l l}
  $x_{m+\alpha}+\sum_{k=j+1}^m a_{(\alpha+n)k}x_{n+k},$ &$\text{for } \beta=j,
   1\leq\alpha\leq n-m$\\
  $x_{m+\alpha}, $&$\text{for } \beta=j, n-m+1\leq\alpha\leq n-m+j$\\
  $x_{m+\alpha}-\sum_{k=j+1}^m a_{(\alpha-n)k}x_{n+k},$ &$\text{for } \beta=j,
   n+1\leq\alpha\leq 2n-m$\\
  $-a_{(\alpha+n)j}x_{n+\beta},$ &$\text{for } \beta>j, 1\leq\alpha\leq n-m$\\
  $x_{n+\beta}, $ &$\text{for } \beta>j, \alpha = n-m+j$\\
  $a_{(\alpha-n)j}x_{n+\beta},$ &$\text{for } \beta>j, n+1\leq\alpha\leq 2n-m$\\
  $0,$ &$\text{elsewhere}.$\\
 \end{tabular}\right.
\]
From this we may compute the matrix $B_{A,x} = D_A\pi_A(x)\left(D_A\pi_A(x)\right)^T$
at $A=0$: The matrix $D_A \pi_A(x)$ is an $m \times (2nm-\frac{m(3m-1)}{2})$-matrix.
The rows correspond to the component functions of the mapping $\pi_A$ and the columns
correspond to all possible pairs $(\alpha,\beta)$. When $A=0$, the entries on the
$j$th row of the matrix $D_A \pi_A(x)$ are
\begin{center}
\begin{tabular}{l l}
$x_{m+\alpha},$ & $\text{for } \beta=j, \alpha \in \{1,...,n-m+j,n+1,...,2n-m\}$\\
$x_{n+\beta},$  & $\text{for } \alpha=n-m+j, \beta \in \{j+1,...,m\}$\\
$0,$            & $\text{elsewhere}.$
\end{tabular}
\end{center}
From this we see that the coordinates $x_{m+1}, \ldots, x_{2n}$ appear on each row
exactly once and the other entries are zero. Thus all the diagonal entries of the
matrix $B_{0,x}$ are $[B_{0,x}]_{ii} = \Sigma_{k=1}^{2n-m} x_{m+k}^2$ for every
$i=1, \ldots, m$.

The entries $x_{m+1}, \ldots, x_{2n}$ appear on different positions on different
rows. If $i<j$, there exists a non-zero entry
at the same column
on both rows $i$ and $j$
if and only if
$\alpha = n-m+i$ and $\beta=j$. The entry on  $i$th row is $x_{n+\beta} = x_{n+j}$
and the entry on $j$th row is $x_{m+\alpha} = x_{n+i}$. This implies that
the non-diagonal entries of the matrix $B_{0,x}$ are $[B_{0,x}]_{ij} =
x_{n+i}x_{n+j}$ for every $i,j = 1, \ldots, m$, $i \neq j$. Thus
\[
\left[ B_{0,x} \right]_{ij} = \left\{
\begin{tabular}{l l}
 $\sum_{k=1}^{2n-m} x_{m+k}^2 =: \Delta_x,$ & $\text{for } i=j$\\
 $x_{n+i}x_{n+j}$ & $\text{for } i \neq j$.\\
\end{tabular}\right.
\]
Using an inductive argument one can see that the determinant of such matrix is
\[
 \det B_{0,x} = \Delta_x^m + \sum_{i = 2}^m (-1)^{i-1}(i-1)\Delta_x^{m-i}
 \sum_{\alpha \in \Lambda(m,i)} x_{\alpha(1)}^2 \cdots x_{\alpha(i)}^2,
\]
where
\[
\Lambda(m,i) = \{\alpha = (\alpha(1), \ldots, \alpha(i)) \pv \alpha(k) \in
\{n+1, \ldots, n+m\} \forall k, \alpha(1) < \ldots < \alpha(i)\}
\]
is the set of all strictly increasing
sequences of length $i$ consisting of integers from the interval $[n+1, n+m]$.

Let $b \in S^{2n-1}$. Note that the entry $\Delta_b$ in the diagonal
of the matrix $B_{0,b}$ is precisely $\Vert P_{V_0^\perp}(b) \Vert^2
= 1-\Vert\Phi_b(0)\Vert^2$. We will show that
\begin{equation}\label{det>c}
 \det B_{0,b} \geq \Delta_b^m - \Delta_b^{m-2} \sum_{\alpha \in
 \Lambda(m,2)} b_{\alpha(1)}^2 b_{\alpha(2)}^2 \geq \frac{1}{2}\Delta_b^m.
\end{equation}

The second inequality is easy:
\[
\begin{aligned}
 \frac{1}{2}\Delta_b^m - \Delta_b^{m-2} \sum_{\alpha \in \Lambda(m,2)}
  b_{\alpha(1)}^2 b_{\alpha(2)}^2
&\geq
 \frac{1}{2}\Delta_b^{m-2} \left( \left( \sum_{i=1}^m b_{n+i}^2 \right)^2 -
 \sum_{\alpha \in \Lambda(m,2)} 2 b_{\alpha(1)}^2 b_{\alpha(2)}^2 \right)\\
&=
 \frac{1}{2}\Delta_b^{m-2} \sum_{i=1}^m b_{n+i}^4\\
&\geq 0.
\end{aligned}
\]

For the first inequality it is enough to show that for any $i \in \{3, \ldots, m-1\}$
it holds that
\[
 (i-1)\Delta_b^{m-i} \sum_{\alpha \in \Lambda(m,i)} b_{\alpha(1)}^2 \cdots
 b_{\alpha(i)}^2 \geq i\Delta_b^{m-i-1} \sum_{\alpha \in \Lambda(m,i+1)}
 b_{\alpha(1)}^2 \cdots b_{\alpha(i+1)}^2.
\]
Using the fact that $\Delta_b \geq \sum_{j=n+1}^{n+m} b_j^2$, we see that the above
inequality holds if
\[
 (i-1)\sum_{\alpha \in \Lambda(m,i)} b_{\alpha(1)}^2 \cdots b_{\alpha(i)}^2 \left(
\sum_{j=n+1}^{n+m} b_j^2 \right)
 \geq i \sum_{\alpha \in \Lambda(m,i+1)} b_{\alpha(1)}^2 \cdots
 b_{\alpha(i+1)}^2.
\]
Now
\begin{align*}
  &\hspace{1.4em}(i-1)\sum_{\alpha \in \Lambda(m,i)} b_{\alpha(1)}^2 \cdots
 b_{\alpha(i)}^2 \left( \sum_{j=n+1}^{n+m} b_j^2 \right)
 -  i \sum_{\alpha \in \Lambda(m,i+1)} b_{\alpha(1)}^2 \cdots
 b_{\alpha(i+1)}^2 \\
 &= (i-1)\sum_{\alpha \in \Lambda(m,i)} b_{\alpha(1)}^2 \cdots b_{\alpha(i)}^2 \left(
\sum_{j=n+1}^{\alpha(i)} b_j^2 \right)
 - \sum_{\alpha \in \Lambda(m-1,i)} b_{\alpha(1)}^2 \cdots
 b_{\alpha(i)}^2 \left( \sum_{k=\alpha(i)+1}^{n+m} b_k^2 \right)\\
 &\begin{aligned} = (i-1) &\sum_{k=n+i}^{n+m} \; \sum_{j=n+1}^{k} \; \sum_{\alpha \in
 \Lambda(k-1,i-1)}
 b_{\alpha(1)}^2 \cdots b_{\alpha(i-1)}^2 b_j^2 b_k^2\\
 - &\sum_{j=n+i}^{n+m-1} \; \sum_{k=j+1}^{n+m} \; \sum_{\alpha \in \Lambda(j-1,i-1)}
 b_{\alpha(1)}^2 \cdots b_{\alpha(i-1)}^2 b_k^2 b_j^2
 \end{aligned}\\
 &\begin{aligned} \geq \sum_{k=n+i+1}^{n+m} b_k^2 &\left(\sum_{j=n+1}^{k} \;
 \sum_{\alpha \in \Lambda(k-1,i-1)} b_{\alpha(1)}^2 \cdots b_{\alpha(i-1)}^2 b_j^2
 \right.\\
 &\left.- \sum_{j=n+i}^{k-1} \; \sum_{\alpha \in \Lambda(j-1,i-1)}
 b_{\alpha(1)}^2 \cdots b_{\alpha(i-1)}^2 b_j^2\right)
 \end{aligned}\\
 &\geq 0,
\end{align*}
which proves the first inequality in \eqref{det>c}. This shows that the inequality
\eqref{trans_pi_0} is valid and finishes the proof that
the family $\mathcal F_h^{2n}(n,m)$ is transversal.

\section{Applications}\label{appl_sect}

Transversality together with Theorem \ref{besfed} imply
that the Besicovitch-Federer projection theorem holds for isotropic projections.

\begin{theorem}\label{besfed_isotropic}
Let $E\subset\mathbb R^{2n}$ be $\hm^m$-measurable with $\hm^m(E)<\infty$.
Then $E$ is purely $m$-unrectifiable, if and only if
$\hm^m(P_V(E))=0$ for $\mu_{n,m}$-almost all $V\in\Gh$.
\end{theorem}

The $\alpha$-energy $I_\alpha(\mu)$ of a measure $\mu$ on $\pr^n$
is defined by
\[
 I_\alpha(\mu)= \int \int |x-y|^{-\alpha} d\mu y \, d\mu x
\]
and the Sobolev dimension of a finite measure $\mu$ on $\pr^n$ is defined as
\[
 \dimS \mu = \sup \left\{ \alpha \in \pr \pv \int(1+|x|)^{\alpha-n}|\hat\mu(x)|^2
  dx < \infty \right\},
\]
where $\hat\mu$ is the Fourier transform of the measure $\mu$.

Transversality immediately yields also the following result concerning the dimension
of projected measures. See \cite[Theorem 7.3]{PS} for more information.

\begin{theorem}\label{PS_7.3}
Let $0 < m \leq n$ and suppose that $\mu$ is a finite positive measure on $\prn$ with
finite $\alpha$-energy for some $\alpha > 0$. If $\sigma \in (0, \alpha]$, then
\[
 \dim \{V \in \Gh \pv \dimS \mu_V < \sigma \} \leq 2nm - \frac{m(3m-1)}{2} + \sigma
 - \max\{\alpha, m\},
\]
where $\mu_V$ is the projection of
the measure $\mu$ onto the $m$-plane $V$, that is, $\mu_V(A) = \mu(P_V^{-1}(A))$
for all $A \subset V$.
\end{theorem}

It follows from the definition of the Sobolev dimension 
that if $0<\dimS \mu\leq n$, then $\dimS \mu = \sup\{\alpha \pv
I_\alpha(\mu)<\infty\}$. 
In particular, if a Borel set $E\subset\pr^n$ supports a probability measure $\mu$
with $\dimS(\mu)\leq n$, then $\dim E \geq \dimS\mu$. If $\dimS\mu>n$, then $\mu$
is absolutely continuous. These facts together with Theorem \ref{PS_7.3} imply the
following result on the dimension of exceptional sets.

\begin{theorem}\label{dim_of_excep}
 Let $n,m$ be integers such that $0<m\leq n$ and let $E \subset \pr^{2n}$ be a Borel
 set with $\dim E = s$.
\begin{itemize}
 \item [(1)] If $s \leq m$, $\dim\{V \in \Gh \pv \dim P_V(E) < s\} \leq
  2nm - \frac{m(3m+1)}{2} + s$.
 \item [(2)] If $s > m$, $\dim\{V \in \Gh \pv \hm^m(P_V(E)) = 0\} \leq
  2nm - \frac{3m(m-1)}{2} - s$.
\end{itemize}
\end{theorem}

\begin{proof}
Assume first that $s > m$. Then by Frostman's lemma
we can take $\alpha > m$ and a probability measure
$\mu$ supported on $E$
such that $I_\alpha(\mu)<\infty$. Now Theorem \ref{PS_7.3} implies that
\begin{align*}
 &\dim\{V\in\Gh\pv\mu_V\;\text{ is not absolutely continuous }\}\\
 &\leq \dim\{V\in\Gh\pv\dimS\mu_V\leq m\}
 \leq 2nm - \frac{3m(m-1)}{2} - \alpha.
\end{align*}
It follows that
\[
 \dim\{V\in\Gh\pv\hm^m(P_V(E))=0\} \leq 2nm - \frac{3m(m-1)}{2} - \alpha.
\]
Letting $\alpha \nearrow s$ implies the claim. The first claim is proven similarly.
\end{proof}

In \cite[Theorem 1.2]{BFMT} it is shown that almost all isotropic projections
onto $m$-planes preserve the Hausdorff dimension for sets whose dimension is
at most $m$. For sets with dimension greater than $m$ almost all projections
have positive $\hm^m$ measure. Their proof uses energy estimates and Frostman's
lemma. Theorem \ref{dim_of_excep} 
strengthens this theorem by providing dimension estimates for the sets of
exceptional parameters.

The behaviour of dimensions of sets and measures under subfamilies of orthogonal
projections has recently been studied
by E. J\"arvenp\"a\"a, M. J\"arvenp\"a\"a and T. Keleti in \cite{JJK} and
D. Oberlin in \cite{O}.
Their results, however, do not give anything new to our setting due to the fact that
the family we are studying is transversal.

Theorems \ref{besfed_isotropic} and \ref{dim_of_excep} yield corresponding results
for the horizontal projections in the Heisenberg group. We denote by
$\pi: \He^n\to\prn, \pi(z,t)=z$ the projection onto the first $2n$ coordinates.

\begin{corollary}\label{besfed_hor}
 Let $E\subset\He^n$ be a Borel set with $\hm_{\rm E}^m(\pi(E))<\infty$.
Then $\hm^m(P_\Hor(E))=0$ for $\mu_{n,m}$-almost all $V\in\Gh$, if and only if
$E \subset A \times \pr$, where $A\subset \prn$ is purely $m$-unrectifiable in the
Euclidean sense.
\end{corollary}

\begin{corollary}\label{dim_hor_proj}
  Let $n,m$ be integers such that $0<m\leq n$ and let $E \subset \He^n$ be a Borel
 set with $\dimH E = s$.
\begin{itemize}
 \item [(1)] If $s \leq m+2$, $\dim\{V \in \Gh \pv \dim P_\Hor(E) < s-2\} \leq
  2nm - \frac{m(3m+1)}{2} + s - 2$.
 \item [(2)] If $s > m+2$, $\dim\{V \in \Gh \pv \hm^m(P_\Hor(E)) = 0\} \leq
  2nm - \frac{3m(m-1)}{2} - s + 2$.
\end{itemize}
\end{corollary}

\begin{proof}
Since $P_\Hor = P_V \circ \pi$, we have by
\cite[Proof of Theorem 1.1]{BFMT} that $\dimE \pi(E) \geq \dimH E -2 = s - 2$.
The rest of the proof is the same as in Theorem \ref{dim_of_excep}.
\end{proof}

\end{document}